\documentclass[12pt,reqno]{amsart}
\usepackage{amscd,amsmath,amsthm,amssymb}
\usepackage{color}
\usepackage{tikz}
\usepackage{url}
\usepackage{circuitikz}
\usepackage{float}
\usetikzlibrary{positioning}

\newtheorem{Theorem}{Theorem}[section]
\newtheorem{Lemma}[Theorem]{Lemma}
\newtheorem{Corollary}[Theorem]{Corollary}
\newtheorem{Proposition}[Theorem]{Proposition}
\newtheorem{Remark}[Theorem]{Remark}

\newtheorem{Example}[Theorem]{Example}

\newtheorem{Question}[Theorem]{Question}

\def\qed{\ifhmode\textqed\fi
	\ifmmode\ifinner\quad\qedsymbol\else\dispqed\fi\fi}
\def\textqed{\unskip\nobreak\penalty50
	\hskip2em\hbox{}\nobreak\hfill\qedsymbol
	\parfillskip=0pt \finalhyphendemerits=0}
\def\dispqed{\rlap{\qquad\qedsymbol}}

\def\height{\textup{height}}

\def\depth{\textup{depth\,}}

\def\pd{\textup{proj\,dim}}
\def\supp{\textup{supp}}
\def\reg{\textup{reg}}

\begin{document}
	
	\title{On homological invariants and Cohen-Macaulayness of closed neighborhood ideals}
	\author{Somayeh Moradi, Leila Sharifan}

	\address{Somayeh Moradi, Department of Mathematics, Faculty of Science, Ilam University, P.O.Box 69315-516, Ilam, Iran}
	\email{so.moradi@ilam.ac.ir}
	
		\address{Leila Sharifan, Department of Mathematics, Hakim Sabzevari University, P.O. Box 397, Sabzevar, Iran}
	\email{l.sharifan@hsu.ac.ir}
	
	
	\begin{abstract}
Let $G$ be a finite simple graph and $NI(G)$ be the closed neighborhood ideal of $G$ in the polynomial ring $S=K[V(G)]$.
In this paper, we study the  Castelnuovo-Mumford regularity,   projective dimension and Cohen-Macaulayness of this ideal. For any chordal graph $G$, we show that
$\text{reg}(S/NI(G))=\tau(G)$, where $\tau(G)$ denotes the vertex cover number of $G$. This generalizes the corresponding result for trees shown in \cite{CJRS}, as in trees $\tau(G)$ is the same as the matching number of $G$.  When $G$ is a bipartite graph or a very well-covered graph, we notice that $\text{reg}(S/NI(G))\geq \tau(G)$ and that this inequality can be strict in general. Moreover, we describe the projective dimension of $S/NI(G)$ for some families of graphs. Finally, we give a characterization of very well-covered graphs $G$ for which the ring $S/NI(G)$ is Cohen-Macaulay.
	\end{abstract}
	
	\maketitle
	\vspace*{-1.8em}
\section{Introduction}

The fruitful interplay between combinatorics and commutative algebra has been a central theme of research over the past several decades. A particularly rich source of interaction arises from the study of squarefree monomial ideals associated with combinatorial objects such as graphs, hypergraphs, and simplicial complexes. Through this correspondence, algebraic invariants of monomial ideals often encode subtle combinatorial properties, while combinatorial techniques provide powerful tools for studying homological and structural aspects of ideals.
A classical and extensively studied example is the edge ideal of a graph, introduced by Villarreal \cite{V}. 
Several other classes and generalizations have since been proposed to capture richer combinatorial structures, including path ideals \cite{B,CD, HeV}, 
$t$-clique ideals \cite{KM,M}, complementary edge ideals \cite{FM,FM1,HQS}, and facet ideals of simplicial complexes \cite{Fa}.

More recently,  the authors introduced the closed neighborhood ideal 
$NI(G)$ of a graph $G$ in \cite{SM}. This ideal is closely related to dominating sets of $G$, as the minimal prime ideals of $NI(G)$ correspond to minimal dominating sets of $G$. This suggests a rich interplay of combinatorics of the graph and algebraic properties of the ideal. The dominance complex of a graph $G$ which is the Stanley-Reisner complex of $NI(G)$ was studied by Matsushita and Wakatsuki \cite{MW}, where they showed that this complex is the Alexander dual of the neighborhood complex of the complement of $G$. The neighborhood complex of a  graph was introduced and used by  Lov$\acute{}$asz in the proof of  Kneser’s conjecture \cite{L}. Both dominance and  neighborhood complexes are fundamental objects in combinatorics and have been studied extensively in the literature.

Let $G$ be a finite simple graph with the vertex set $V(G) =[n]$ and let $S=K[x_1,\ldots,x_n]$ be the polynomial ring in $n$ variables over a field $K$. The {\em closed neighborhood ideal} of $G$ is a squarefree monomial ideal in $S$, defined as
$NI(G)=(\prod_{j\in N_G[i]} x_j: i\in [n]),$ where   
$N_G[i]$ denotes the closed neighborhood of $i$ in $G$. 
Since their introduction, considerable effort has been devoted to analysing closed neighborhood ideals from both algebraic and combinatorial perspectives, with particular emphasis on their structural and homological properties. Moreover, recently some related ideals, namely  open neighborhood ideals of graphs and  $t$-closed neighborhood ideals  were introduced and studied by Lim, et al. \cite{Lim,Lim1}, and Sharifan \cite{S1}, respectively.

It was shown in \cite{SM} that the matching number $a(G)$ of graph $G$ is a lower bound for $\reg(S/NI(G))$, provided that $G$ is a tree, and it was conjectured that this inequality is indeed an equality.
This result was refined in two ways by Chakraborty, et al. in \cite{CJRS}, where they proved the conjecture, and further showed that the inequality holds true for any graph.
Further combinatorial descriptions for the regularity or projective dimension of closed neighborhood ideals  for some families of graphs like  $m$-book graph and complete multipartite graphs can be found in \cite{SM}, and for unicyclic graphs and forests in \cite{CJRS}. 
 Joseph, Roy and Singh \cite{J} studied minimal free resolutions of closed neighborhood ideals within the framework of Barile-Macchia resolutions, and showed that for any tree $T$, the Barile-Macchia resolution of $NI(T)$  is minimal. Hien  and Vu \cite{HV} studied associated primes of the second power of closed neighborhood ideals. 
In \cite{N1,N2,N3}, Nasernejad, et al. studied the normally torsion-free property, normality and  persistence property of closed neighborhood ideals. The Cohen–Macaulayness of closed neighborhood ideals of trees was studied by Honeycutt and Sather-Wagstaff in \cite{HS} and of chordal graphs and bipartite graphs by Leaman in \cite{L}, where complete combinatorial characterizations were given. 
Despite these developments, many fundamental questions concerning the homological behaviour of closed neighborhood ideals remain open. In particular, the relationship between the Castelnuovo–Mumford regularity, projective dimension, and Cohen–Macaulayness of 
$S/NI(G)$ and the combinatorial structure of 
$G$ is not yet fully understood. The purpose of this paper is to further explore these connections, with special emphasis on chordal graphs and very well-covered graphs, and to provide new combinatorial interpretations of algebraic invariants associated with closed neighborhood ideals.

The paper is structured as follows. In Section 2, we recall some basic notation and terminology from commutative algebra and combinatorics that will be needed in the paper.
In Section 3, we  study the regularity and the projective dimension of closed neighborhood ideals. As one of the main results, in Theorem \ref{chordalReg}, we prove that for any chordal graph $G$, we have $\reg(S/NI(G))=\tau(G)$, where 
$\tau(G)$ denotes the vertex cover number of $G$. When $G$ is a tree, it is known by König's theorem that $\tau(G)$ is the same as the matching number $a(G)$ of $G$. Hence, Theorem \ref{chordalReg} gives a generalization of the equality $\reg(S/NI(G))=a(G)$, proved for trees in  \cite[Theorem 1.1]{CJRS}. Another family of graphs for which the equality $\reg(S/NI(G))=\tau(G)$ holds, is given in Theorem \ref{ISGen}. In Example \ref{lu}, we see that in general $\tau(G)$ is neither an upper bound nor a lower bound for the regularity of $S/NI(G)$, as for any complete bipartite graph $G=K_{m,n}$ with $m,n>1$ and $(m,n)\neq (2,2)$, we have  $\reg(S/NI(G))>\tau(G)$ and for the $5$-cycle $C_5$ we have $\reg(S/NI(C_5))<\tau(C_5)$.  However, if $G$ is a  bipartite graph or a very well-covered graph, then we have $\reg(S/NI(G))\geq \tau(G)$,
see Proposition \ref{ltau}. In Proposition \ref{bounds}, we show that the independence number $\alpha(G)$ of $G$ is a lower bound for the projective dimension of $S/NI(G)$, and in Theorem \ref{ISGen} and Proposition \ref{lPd} we present  families of graph, for which the projective dimension achieves this bound. Finally, in Theorem \ref{multiplicity} a description of the multiplicity of  $S/NI(G)$ in terms of the dominating sets of $G$ is given.

The main result of Section 4 is Theorem \ref{VW-CM}, where we give a combinatorial characterization of the Cohen-Macaulayness of $S/NI(G)$ for very well-covered graphs. To this aim, we use a characterization of well-dominated graphs for which the height of $NI(G)$ is half of the number of  vertices of $G$, given in \cite{L}. Theorem \ref{VW-CM} helps to recover the characterization of bipartite graphs for which $S/NI(G)$ is  Cohen-Macaulay, proved by Leaman in \cite{L}. Moreover, using a result due to Topp and Volkmann \cite{TV}, we are able to add the equivalent condition $\gamma(G)=\alpha(G)$ to the characterization given by Leaman. The equality $\gamma(G)=\alpha(G)$ is in general only a necessary condition for the Cohen-Macaulayness of $S/NI(G)$.   Finally, we highlight a purely combinatorial application of another result by Leaman \cite{TV} on the characterization of the Cohen-Macaulayness of $S/NI(G)$ for chordal graphs. More precisely, in Corollary \ref{alphaGamma} we give a classification of chordal graphs $G$ satisfying $\gamma(G)=\alpha(G)$.
This generalizes the analogous result proved for block graphs in \cite{TV}.

 \section{Preliminaries}
 
 In this section we recall the notation and terminology from commutative algebra and combinatorics, that will be used in the paper. 
 Throughout, $K$ will denote an infinite field and $S= K[x_1,\ldots,x_n]$ will be a polynomial ring over $K$.

Let $G$ be a finite simple graph with the vertex set $V(G)=[n]$ and the edge set $E(G)$.
For any vertex $i\in [n]$, the {\em open  neighborhood} of $i$ is the set $N_G(i)=\{j\in [n]:\, \{i,j\}\in E(G)\}$, and the {\em closed neighborhood} of $i$ is $N_G[i]=N_G(i)\cup\{i\}$. 
A subset $D\subseteq V(G)$ is called a \textit{dominating set} of $G$, if $D\cap N_G[i]\neq \emptyset$  for any vertex $i$ of $G$.  Also a dominating set $D$ of $G$ is called a \textit{minimal dominating set} of $G$, if no proper subset of $D$ is a dominating set of $G$. We call $G$ {\em well-dominated}, if all minimal dominating sets of $G$ have the same size.
  The \textit{domination number} of $G$ denoted by $\gamma(G)$, is the minimum size of a dominating set of $G$. Moreover, we set $$\gamma_0(G)=\max\{|D|: D \textrm{ is a minimal dominating set of } G\}.$$

A {\em vertex cover of $G$} is a subset $C\subseteq V(G)$ such that $C\cap e\neq \emptyset$ for any $e\in E(G)$. A vertex cover $C$ is called a {\em minimal vertex cover} of $G$, if no proper subset of $C$ is a vertex cover of $G$. The minimum size of a   vertex cover of $G$ is called the {\em vertex cover number} of $G$ and is denoted by $\tau(G)$.  
A subset $A\subseteq V(G)$ is called an {\em independent set} of $G$, if $A$ contains no edge of $G$. A {\em maximal independent set} of $G$ is an independent set which is maximal with respect to inclusion, among maximal independent sets of $G$.  The maximum size of an independent set  of $G$ is called the {\em independence number} of $G$ and is denoted by $\alpha(G)$. 
It is easily seen that a subset $A$ is an independent set of $G$ if and only if $V(G)\setminus A$ is a vertex cover of $G$. Thus $\alpha(G)=n-\tau(G)$.  
A graph $G$ is called {\em well-covered}, if all maximal independent sets of $G$ have the same cardinality. It follows from definitions that
\begin{Remark}\label{covered}
	Any maximal independent set of $G$ is a minimal dominating set of $G$. Hence, $\gamma(G)\leq \alpha(G)$, and any well-dominated graph is well-covered.
\end{Remark}

A subset $M$ of $E(G)$ is called a \textit{matching} for $G$, if any two edges in $M$ are disjoint and a matching $M$ is called a \textit{maximal matching} of $G$ if it is not properly contained in another matching of $G$. A {\em perfect matching} of $G$ is a matching that covers $V(G)$, i.e., any vertex of $G$ is an endpoint of an edge in the matching. The \textit{matching number} of $G$ is the maximum size of a matching in $G$ and is denoted by $a(G)$. It follows from the definitions that $a(G)\leq \tau(G)$  for any graph $G$. A graph $G$ is called a {\em König graph}, if $\tau(G)=a(G)$. In  1931, König showed that any bipartite graph is a König graph.

A {\em clique} of $G$ is a subset $B$ of $V(G)$ such that the induced subgraph $G[B]$ of $G$ on the vertex set $B$ is a complete graph. A clique $B$ is a {\em maximal clique} of $G$, if it is not properly contained in another clique of $G$. 
A vertex $i\in V(G)$ is called a {\em free vertex} of $G$, if it belongs to precisely one maximal clique of $G$.  A {\em leaf} of the graph $G$ is a vertex of degree one.
A graph $G$ is called {\em chordal}, if any cycle of length $m\geq 4$ in $G$ has a chord, which is an edge that is not part of the cycle but connects two vertices of the cycle.
A vertex $i$ of $G$ is called a {\em simplicial vertex} of $G$, if $N_G[i]$ is a clique. By a result of Dirac (1961), it is known that every chordal graph has a simplicial vertex. A {\em tree} is a connected graph with no cycles. A {\em spanning tree} of a graph $G$ is a  subgraph $T$ of $G$ such that $T$ is a tree and $V(T)=V(G)$.

A {\em very well-covered} graph is a well-covered graph without isolated vertices such that the size of any minimal vertex cover  (maximal independent set) is half of the number of vertices of the graph.
Any very well-covered graph $G$ on $2n$ vertices has a perfect matching of size $n$. Hence, it is a König graph with $\tau(G)=a(G)=n$.

The graph $G$ is called a {\em whisker graph}  if every vertex of $G$ has degree one, or has a unique degree-one neighbor.  
For a subset $W\subseteq V(G)$, by $G\setminus W$ we mean a graph obtained from $G$ by removing the vertices in $W$ and all the edges adjacent to a vertex in $W$. If $W=\{i\}$, then we denote $G\setminus W$ simply by $G\setminus i$. For a subset $\{i,j\}\subset V(G)$, by $G\cup \{\{i,j\}\}$, we mean a graph obtained from $G$ by adding the edge $\{i,j\}$ to   $G$.

Let $G$ be a graph with the vertex set $[n]$, and let $S=K[x_1,\ldots,x_n]$ be the polynomial ring over a field $K$. For any vertex $i$, we set $u_{i,G}=\prod_{j\in N_G[i]} x_j$ to be a monomial in $S$.
The {\em closed neighborhood ideal} of $G$ is a squarefree monomial ideal of $S$ defined as
$$NI(G)=(u_{i,G}: i\in V(G)).$$  
By \cite[Lemma 2.2]{SM}, the minimal prime ideals of $NI(G)$ are of the form $P_D=(x_i:\, i\in D)$, where $D$ is a minimal dominating set of $G$. So,  $\dim(S/NI(G))=n-\gamma(G)$.

 	\section{Homological invariants of closed neighborhood ideals}\label{sec1}
 	
  In this section, we study the Castelnuovo-Mumford regularity, projective dimension and multiplicity of the ring $S/NI(G)$.
 	When $G$ is a tree, it was shown in \cite{CJRS} that $\reg(S/NI(G))=a(G)$. Since any tree is a bipartite graph, by a result of König it follows that for any tree $G$, the equality $a(G)=\tau(G)$ holds, and hence 	$\reg(S/NI(G))=\tau(G)$. In Theorem \ref{chordalReg}, we generalize this result to any chordal graph. To this aim we use the following lemma.

 	\begin{Lemma}\cite[Corollary 4.8]{CHHK} \label{reg}
 	Let $I\subset S$ be a monomial ideal and $x$ be a variable of $S$. Then
 \begin{enumerate}
 	\item [(i)] $\reg(S/I)\leq \max\{\reg(S/(I:x))+1,\reg(S/(I,x))\}$.
 	\item [(ii)] $\reg(S/(I,x))\leq \reg(S/I)$.
 \end{enumerate}
 	\end{Lemma}
 	
 	\begin{Theorem}\label{chordalReg}
 		Let $G$ be a chordal graph. Then $\reg(S/NI(G))=\tau(G)$.
 	\end{Theorem}
 	
 	\begin{proof}
  We prove
 	the statement by induction on $n=|V(G)|$. Let  $z$ be a simplicial vertex of $G$, and let  $N_G(z)=\{x_1,\ldots,x_r\}$. Then $u_{z,G}=zx_1\cdots x_r$. Since $N_G[z]\subseteq N_G[x_i]$ for all $i$, we have $u_{z,G}| u_{x_i,G}$ for all $i$. Hence,
 	$$NI(G)=(u_{z,G})+(u_{y,G}:\, y\in V(G)\setminus\{z,x_1,\ldots,x_r\}).$$
 	By Lemma~\ref{reg}(i), we have 
\begin{equation}\label{x_1}
	 \reg(S/NI(G))\leq \max\{\reg(S/(NI(G):x_1))+1,\reg(S/(NI(G),x_1))\}.
\end{equation} 	
 Set $G_i=G\setminus x_i$ for any $1\leq i\leq r$.
 	Notice that $(u_{z,G}:x_i)=zx_1\cdots x_{i-1}x_{i+1}\cdots  x_r=u_{z,G_i}$. Moreover, for any vertex $y\in V(G)\setminus\{z,x_1,\ldots,x_r\}$,  if  $y\in N_G(x_i)$, then we have $(u_{y,G}:x_i)=u_{y,G}/x_i=u_{y,G_i}$, and if $y\notin N_G(x_i)$, then $(u_{y,G}:x_i)=u_{y,G}=u_{y,G_i}$. 
 	Therefore, $$(NI(G):x_i)=(u_{z,G_i})+(u_{y,G_i}:\, y\in V(G_i)\setminus\{z,x_1,\ldots,x_r\})=NI(G_i).$$  
 	By induction hypothesis, $\reg(S/(NI(G):x_i))=\reg(S/(NI(G_i)))=\tau(G_i)$.
 	 We show that $\tau(G_i)+1=\tau(G)$ for any $1\leq i\leq r$. Let $C$ be a minimal vertex cover of $G_i$ with $\tau(G_i)=|C|$. Then $C\cup\{x_i\}$ is a vertex cover of $G$. This means that $\tau(G)\leq |C|+1=\tau(G_i)+1$.
 	 Now, let $D$ be a minimal vertex cover of $G$ with $\tau(G)=|D|$. If $x_i\in D$, then $D\setminus \{x_i\}$ is a vertex cover of $G_i$. Thus we obtain $\tau(G_i)\leq |D|-1=\tau(G)-1$. If $x_i\notin D$, then $$\{z,x_1,\ldots,x_{i-1},x_{i+1},\ldots,x_r\}\subseteq N_G(x_i)\subseteq D.$$ So $D'=(D\setminus\{z\})\cup \{x_i\}$ is a vertex cover of $G$ with $x_i\in D'$ and $|D'|=|D|=\tau(G)$. Thus the same argument as in the case $x_i\in D$ implies that $\tau(G_i)+1\leq\tau(G)$.  Hence, $\tau(G_i)+1=\tau(G)$ for all $i$, as claimed. This together with (\ref{x_1}) and the equality $\reg(S/(NI(G):x_1))=\tau(G_1)$ gives 
 	 \begin{equation}\label{eq:1}
 	 	\reg(S/NI(G))\leq \max\{\tau(G),\reg(S/(NI(G),x_1))\}.
 	 \end{equation} 	
 	  
 \smallskip
 
Next, we show that $\reg(S/(NI(G),x_1))\leq \tau(G)$. For any $1\leq i\leq r$, we set $J_i=(NI(G),x_1,\ldots,x_i)$. Applying Lemma \ref{reg}(i) to $J_i$ we have  
 \begin{equation}\label{J_i}
 	\reg(S/J_i)\leq \max\{\reg(S/(J_i:x_{i+1}))+1,\reg(S/J_{i+1})\}.
 \end{equation}

For each $1\leq i\leq r-1$, we have $$(J_i:x_{i+1})=((NI(G):x_{i+1}),x_1,\ldots,x_i)=(NI(G_{i+1}),x_1,\ldots,x_i),$$ and hence by Lemma \ref{reg}(ii), $$\reg(S/(J_i:x_{i+1}))=\reg(S/(NI(G_{i+1}),x_1,\ldots,x_i))\leq \reg(S/NI(G_{i+1}))=\tau(G_{i+1}),$$
Using this, the equality $\tau(G_{i+1})+1=\tau(G)$ and the inequalities (\ref{J_i}), we obtain
 \begin{equation}\label{inductIneq}
\reg(S/J_i)\leq \max\{\tau(G),\reg(S/J_{i+1})\}
 \end{equation} 	
for any $1\leq i\leq r-1$. 
Notice that $J_r=(NI(G),x_1,\ldots,x_r)=(NI(G\setminus  z),x_1,\ldots,x_r)$. 
So by Lemma \ref{reg}(ii) and induction hypothesis, 
$$\reg(S/J_r)\leq \reg(S/(NI(G\setminus z))=\tau(G\setminus  z))\leq \tau(G).$$ Applying this and multiple using of inequalities (\ref{inductIneq}), we conclude that $\reg(S/J_1)\leq \tau(G)$. Combining this with (\ref{eq:1}) we get $\reg(S/NI(G))\leq \tau(G)$.  

To complete the proof, we need to show that $\reg(S/NI(G))\geq \tau(G)$. Set $$L=(u_{y,G}:\, y\in V(G)\setminus\{z,x_1,\ldots,x_r\}).$$ We have  $NI(G)=(u_{z,G})+L.$ Since $z$ does not divide $u_{y,G}$ for any $y\in V(G)\setminus\{z,x_1,\ldots,x_r\}$, it follows from \cite[Corollary 2.6]{S} that the minimal free resolution of $NI(G)$ is obtained by the mapping cone of a complex homomorphism which is a lifting of the map $S/(L:u_{z,G})\rightarrow S/L$. So by \cite[Theorem 2.4]{S}(b), we have $$\reg(S/NI(G))=\max\{\reg(S/L),\reg(S/(L:u_{z,G}))+r\}.$$    
Hence, $\reg(S/NI(G))\geq \reg(S/(L:u_{z,G}))+r$.
It is easy to see that $(L:u_{z,G})=NI(G\setminus N_G[z])$.
Thus by induction hypothesis, $\reg(S/(L:u_{z,G}))=\tau(G\setminus N_G[z])$. Let $C$ be a minimal vertex cover of $G\setminus N_G[z]$ with $\tau(G\setminus N_G[z])=|C|$. Then clearly $C\cup\{x_1,\ldots,x_r\}$ is a vertex cover of of $G$. So $\tau(G)\leq |C|+r=\tau(G\setminus N_G[z])+r$. Combining the above inequalities we have $$\reg(S/NI(G))\geq \reg(S/(L:u_{z,G}))+r=\tau(G\setminus N_G[z])+r\geq \tau(G).$$
The proof is complete.
\end{proof}

The following example shows that Theorem \ref{chordalReg} does not hold, if we drop the assumption that $G$ is chordal. Moreover, it shows that in general $\tau(G)$ is neither an upper bound nor a lower bound for $\reg(S/NI(G))$. However, we will see in Proposition \ref{ltau} that for the family of bipartite graphs and very well-covered graphs, $\tau(G)$ is a lower bound for the regularity of $S/NI(G)$.

\begin{Example}\label{lu}
{\em (i)} Let $G=K_{m,n}$ be a complete bipartite graph with $m, n\geq 2$. Then by \cite[Theorem 2.10(ix)]{SM}, we have $$\reg(S/NI(G))=|V(G)|-2=m+n-2,$$ while $\tau(G)=\min\{m,n\}$. So 
$\reg(S/NI(G))=\tau(G)$ if and only if $m=n=2$. Therefore,
$\reg(S/NI(G))>\tau(G)$ for $(m,n)\neq (2,2)$. 

\medskip

{\em (ii)} Let $G$ be the  five cycle $C_5$. Computation  by Macaulay2 shows that $\reg(S/NI(G))=2$, while $\tau(G)=3$. Hence,
$\reg(S/NI(G))<\tau(G)$. 
\end{Example}

 \begin{Proposition}\label{ltau}
If $G$ is a bipartite graph or a very well-covered graph, then we have $\reg(S/NI(G))\geq \tau(G)$.
 \end{Proposition}
 
 \begin{proof}
 It was shown in \cite[Theorem 1.2]{CJRS} that $\reg(S/NI(G))\geq a(G)$ for any graph $G$. Since bipartite graphs and very well-covered graphs are König graphs, we have $a(G)=\tau(G)$. This shows the desired inequality.
 \end{proof}	
 	
 The following proposition gives sharp  bounds for the projective dimension and depth of $S/NI(G)$ for an arbitrary graph $G$. We next discuss some cases, where the inequalities are indeed equalities.
	
	\begin{Proposition}\label{bounds}
		For any graph $G$ we have 
		\begin{enumerate}
			\item [(a)]  $\pd(S/NI(G))\geq \gamma_0(G)\geq \alpha(G)$. 
			\item [(b)] $\depth(S/NI(G))\leq \tau(G)$.
		\end{enumerate}    
	\end{Proposition}
	
	\begin{proof}
(a) The inequality  $\pd(S/NI(G))\geq \gamma_0(G)$ holds by \cite[Corollary 2.4]{SM}. Since any maximal independent set of $G$ is a minimal dominating set of $G$, it follows that $\gamma_0(G)\geq \alpha(G)$.   

\medskip
	
(b) follows from the Auslander-Buchsbaum formula and the equality $n-\alpha(G)=\tau(G)$.
	\end{proof}
	
Theorem \ref{ISGen} and Proposition \ref{lPd} present families of graphs for which the inequalities of Proposition \ref{bounds} become equalities.	
We say $NI(G) $ has an {\em independent set of generators}, if there exists an independent set $A$ of $G$ such that $NI(G)=(u_{i,G}:\, i\in A)$. We denote the minimal set of monomial generators of a monomial ideal $I$ by $\mathcal{G}(I)$. 

\begin{Theorem}\label{ISGen}
	Let $G$ be a graph such that $NI(G)$ has an independent set of generators. Then
	\begin{enumerate}
		\item [(a)] $\pd(S/NI(G))=\alpha(G)=|\mathcal{G}(NI(G))|$.
		\item [(b)] $\reg(S/(NI(G))=\tau(G)=n-|\mathcal{G}(NI(G))|$.
	\end{enumerate}
\end{Theorem}
\begin{proof}
	(a)
	Let  $A$ be an independent set of $G$ and $NI(G)=(u_{x_i,G}:\, x_i\in A)$. Then $V(G)\setminus A$ is a redundant set of vertices in G, and by \cite[Corollary 5.5.18]{L}, $\depth(S/NI(G))\geq n-|A|$. Using this and Proposition \ref{bounds}, we have $$\alpha(G)\leq \pd(S/NI(G)\leq |A|\leq \alpha(G).$$ So the conclusion follows.
	
	(b) By induction on the number of vertices of $G$ we prove the assertion.  Let $A=\{z_1,\ldots,z_r\}$ be as (a). If $n=|A|$, then $G$ has no edge and the result is clear. So, assume that $x\in V(G)\setminus A$ and $u_{z_1,G}|u_{x,G}$.  Set $H=G\setminus x$. Then one can easily see that 
	$$(NI(G):x)=NI(H)=(u_{{z_1},H},\ldots,u_{z_r,H}).$$
	So, $NI(H)$ has an independent set of generators, and by the induction hypothesis,
	$$
	\reg\bigl(S/NI(H)\bigr)=\tau(H)=n-1-\lvert \mathcal{G}(NI(H))\rvert
	= n-1-\lvert \mathcal{G}(NI(G))\rvert .
	$$
	
	Suppose that $1\le \ell \le r$, $x\in \supp(u_{z_j})$ for all $j\le \ell$, and
	$x\notin \supp(u_{z_j})$ for all $j>\ell$. Then
	$$
	NI(G)+(x)
	= (u_{z_{\ell+1},G},\ldots,u_{z_r,G},x)
	= (u_{z_{\ell+1},H},\ldots,u_{z_r,H},x).
	$$
	
	So, $\reg(S/NI(G)+(x))=\reg(S/(u_{z_{\ell+1},H},\ldots,u_{z_r,H}))\leq \reg\bigl(S/NI(H)\bigr)$. The last inequality follows from the fact that the minimal free resolution of $S/NI(H)$ is computed by iteration mapping cone for any order of generators, see \cite[Corollary 2.6]{S}.  Now, by Lemma \ref{reg}, $\reg\bigl(S/NI(G)\bigr)\leq n-\lvert \mathcal{G}(NI(G))\rvert=\tau(G)$.
	
	To show the other side of the inequality, let $L=(u_{{z_2},G},\ldots,u_{z_r,G})$. Then  $$(L:u_{z_1,G})=NI(G\setminus N_G[z_1])=(u_{{z_2},G\setminus N_G[z_1]},\ldots,u_{z_r,G\setminus N_G[z_1]}).$$ Since $NI(G\setminus N_G[z_1])$ has an independent set of generators, by induction hypothesis, $\reg(S/NI(G\setminus N_G[z_1])=\tau(G\setminus N_G[z_1]))=n-\deg_G(z_1)-1-(r-1)=n-\deg_G(z_1)-r$.  On the other hand, by \cite[Theorem 2.4]{S} $$\reg(S/NI(G))=\max\{\reg(S/L),\reg(S/(L:u_{z_1,G}))+\deg_G(z_1)\}.$$
	So $\reg(S/NI(G)\geq \reg(S/(L:u_{z_1,G}))+\deg_G(z_1)  =n-\deg_G(z_1)-r+\deg_G(z_1)=n-r=\tau(G) $.
\end{proof}	

The class of graphs given in Theorem \ref{ISGen}, includes the family of chordal graphs, bipartite graphs and very well-covered graphs other than $C_4$, for which $S/NI(G)$ is Cohen-Macaulay, see Theorem  \ref{chordalCM}, Theorem \ref{bipartiteCM}, and Theorem \ref{VW-CM}, respectively. However, this class is bigger and not any graph in this class has a Cohen-Macaulay ring $S/NI(G)$. In the following we give   examples.

\begin{Example}
{\em (i)} Let $G$ be the following bipartite graph. Then  the set $$A=\{5,6,7,8,9,10,11,12\}$$ is an independent set of $G$ which produces an independent set of generators for $NI(G)$. However, by Theorem \ref{bipartiteCM}, the ring $S/NI(G)$ is not Cohen-Macaulay,  since $G$ it is not a whisker graph.  
	
	
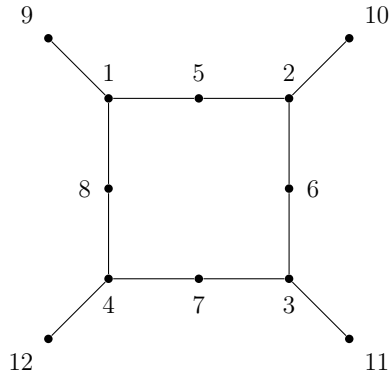
\begin{center}
		\begin{tikzpicture}[scale=0.8, transform shape,
			vertex/.style={
				circle,
				fill=black,
				inner sep=0pt,
				minimum size=4pt
			}
			]
			
			\node[vertex] (v1) at (0,3) {};
			\node[vertex] (v2) at (3,3) {};
			\node[vertex] (v3) at (3,0) {};
			\node[vertex] (v4) at (0,0) {};
			
			\node[vertex] (v5) at (1.5,3) {};
			\node[vertex] (v6) at (3,1.5) {};
			\node[vertex] (v7) at (1.5,0) {};
			\node[vertex] (v8) at (0,1.5) {};
			
			\node[vertex] (v9)  at (-1,4)   {};
			\node[vertex] (v10) at (4,4)    {};
			\node[vertex] (v11) at (4,-1)   {};
			\node[vertex] (v12) at (-1,-1)  {};
			
			\node[above  =2pt of v1] {$1$};
			\node[above  =2pt of v2] {$2$};
			\node[below  =2pt of v3] {$3$};
			\node[below =2pt of v4] {$4$};
			
			\node[above=2pt of v5] {$5$};
			\node[right=2pt of v6] {$6$};
			\node[below=2pt of v7] {$7$};
			\node[left=2pt of v8] {$8$};
			
			\node[above left=2pt of v9] {$9$};
			\node[above right=2pt of v10] {$10$};
			\node[below right=2pt of v11] {$11$};
			\node[below left=2pt of v12] {$12$};
			
			\draw (v1) -- (v5) -- (v2);
			\draw (v2) -- (v6) -- (v3);
			\draw (v3) -- (v7) -- (v4);
			\draw (v4) -- (v8) -- (v1);
			
			\draw (v1) -- (v9);
			\draw (v2) -- (v10);
			\draw (v3) -- (v11);
			\draw (v4) -- (v12);
			
		\end{tikzpicture}
	\end{center}
	More generally, if $G$ is a bipartite graph with the bipartition $X\sqcup Y$ and $G'$ is obtained from $G$, by adding a whisker to every vertex in $X$, then $NI(G)$ has an independent set of generators. 

{\em (ii)} A windmill graph is a graph obtained by joining some complete graphs at a single shared vertex. For any windmill graph $G$, the ideal $NI(G)$ has an independent set of generators, while $S/NI(G)$ is not Cohen-Macaulay.

\end{Example} 

\begin{Proposition}\label{lPd}
	Let $G$ be one of the following graphs:
	\begin{enumerate}
		\item [(a)] A complete multipartite graph.
		\item [(b)] A graph $G$ such that $S/NI(G)$ is Cohen-Macaulay.
	\end{enumerate}
	Then $\pd(S/NI(G))=\alpha(G)$ and $\depth(S/NI(G))=\tau(G)$. 
\end{Proposition}

\begin{proof}
	(a). Let $G=K_{n_1,\ldots,n_r}$ be a complete multipartite graph with $n_1\leq \cdots\leq n_r$. Then it was shown in \cite[Theorem 2.10(viii)]{SM} that $\pd(S/NI(G))=n_r$. Clearly, $\alpha(G)=n_r$. So the equality holds.
	Moreover, we have $\depth(S/NI(G))=n-\pd(S/NI(G))=n-n_r=\tau(G)$.	 
	
	\medskip
	
	(b). Let $p=\pd(S/NI(G))$. By Proposition \ref{bounds} and our assumption, we have $$\alpha(G)\leq p=n-\depth(S/NI(G))=n-\dim(S/NI(G))=\gamma(G)\leq \alpha(G).$$ Hence, we obtain $p=\alpha(G)=\gamma(G)$.   		
\end{proof}

\begin{Remark}\label{C5}
	We note that the inequalities of Proposition \ref{bounds} may be strict in general.
	 Let $G=C_5$. Then $\pd(S/NI(G))=3>2=\gamma_0(G)=\alpha(G)$, and $\depth(S/NI(G))=2<3=\tau(G)$. 
\end{Remark}

For any tree $T$, the equality $\pd(S/NI(T))=\alpha(T)$ was proved in \cite[Theorem 3.14]{J}. We expect that this holds more generally for any chordal graph. Indeed, computations in Macaulay2 motivate the following question.

\noindent 
\begin{Question}
Let $G$ be a chordal graph or a bipartite graph.	Is it true that $\pd(S/NI(G))=\gamma_0(G)$, or more strongly that  $\pd(S/NI(G))=\alpha(G)$?
\end{Question}

 The following result describes the multiplicity of the ring $S/NI(G)$ in terms of dominating sets of $G$. We denote the number of minimal dominating sets of a graph $G$ by $d(G)$, and the multiplicity of a ring $R$ by $e(R)$.

\begin{Theorem}\label{multiplicity}
	For any graph $G$, we have 
	
	\begin{enumerate}
		\item [(a)] $e(S/NI(G))$ is equal to the number of minimal dominating sets of $G$ of cardinality $\gamma(G)$. In particular,
		$e(S/NI(G))\leq d(G)$ and equality holds if and only if $G$ is well-dominated.
		\item [(b)] If
		$V(G)$ is the disjoint union of  maximal cliques $C_1,\ldots,C_r$ of $G$ which admit a free vertex, then
		  $e(S/NI(G))=d(G)=\prod_{i=1}^r|C_i|$. In particular, this equality holds when  $G$ is a well-covered chordal graph, or a whisker graph.
	\end{enumerate}  
	
\end{Theorem}

\begin{proof}
	(a) Let $\Delta=\Delta_{NI(G)}$ be the Stanley-Reisner simplicial complex of the ideal $NI(G)$. 
	Then $F\in \Delta$ if and only if $N_G[x]\nsubseteq F$ for any $x\in V(G)$. This is equivalent to $F^c\cap N_G[x]\neq \emptyset$ for any $x\in V(G)$, where $F^c=V(G)\setminus F$. So $$\Delta=\langle F^c: F \textrm{ is a minimal dominating set of } G\rangle.$$

	By \cite[Corollary 6.2.3.]{HHBook},
	$e(S/NI(G))$ is equal to the number of facets $F^c$ of $\Delta$ with $|F^c|=\dim(\Delta)+1$. Notice that $\dim(\Delta)=\dim(S/NI(G))-1=n-\gamma(G)-1$. So  $e(S/NI(G))$ is the number of facets $F^c$ of $\Delta$ with $n-|F|=n-\gamma(G)$. This shows the first statement. The second statement follows easily from the first statement.
	
	\medskip
	
	(b) The assumption on $V(G)$ implies that $NI(G)$ is an unmixed ideal, see \cite[Corollary 2.11]{S1}.  Therefore $G$ is well-dominated. Thus by (a), we have $e(S/NI(G))=d(G)$. Let $C_1,\ldots,C_r$ be the maximal cliques of $G$ which admit a free vertex. Then any minimal dominating set of $G$ is of the form $\{x_{i_1},\ldots,x_{i_r}\}$, where $x_{i_j}\in C_j$ for all $j$. Moreover, for any such choices of $x_{i_j}$ we get a  minimal dominating set of $G$. So the number of minimal dominating sets of $G$ is $\prod_{i=1}^r|C_i|$. 
	If $G$ is a well-covered chordal graph,    
	or a whisker graph, then $V(G)$ has such a partition to disjoint union of maximal cliques admitting a free vertex, see Corollary \ref{alphaGamma}. 
\end{proof}

\section{Cohen-Macaulayness of very well–covered graphs}

In this section, we give a characterization of very well-covered graphs $G$ for which the ring $S/NI(G)$ is Cohen-Macaulay. 
In order to do this, we need some preparation. The following result which was proved in \cite{L} will be of use.

\begin{Proposition}\cite[Proposition 5.3.3]{L}\label{VWWD}
	Let $G$ be a connected graph on $2n$ vertices such that $\dim(S/NI(G))=n$. Then $G$ is
	well-dominated if and only if every spanning tree of $G$ is a whisker graph. 
\end{Proposition}

The following lemma is an easy observation and a well-known fact on very well–covered graphs. For the reader's convenience we provide a short proof of it. 

\begin{Lemma}\label{part}
	Let $G$ be a very well–covered graph  with $2n$ vertices. Then 	there exists a partition of $V(G)=X\sqcup Y$, such that $X =\{x_1,\ldots,x_n\}$ is minimal vertex cover of $G$, $Y =\{y_1,\ldots,y_n\}$ is a maximal independent set of $G$ and $\{\{x_1,y_1\},\ldots,\{x_n,y_n\}\}\subseteq E(G)$.
\end{Lemma}

\begin{proof}
	It was shown in \cite[Theorem 1.2]{F} that any very well–covered graph $G$ has a perfect matching. Let $\{e_1,\ldots,e_n\}\subseteq E(G)$ be a perfect matching of $G$ with $e_i=\{x_i,y_i\}$ for all $i$. Let $I(G)$ be the edge ideal of $G$. Since $\height(I(G))=n$, it follows that $G$ has a minimal vertex cover $X$ of cardinality $n$. So $X\cap e_i\neq \emptyset$ for all $i$. Without loss of generality, we may assume that $x_i\in X$ for all $1\leq i\leq n$. Since $|X|=n$, we get $X=\{x_1,\ldots,x_n\}$. Thus $Y=X^c=\{y_1,\ldots,y_n\}$ is a maximal independent set of $G$.
\end{proof}

The following lemma shows that Proposition \ref{VWWD} can be applied to very well–covered graphs. 

\begin{Lemma}\label{dimVW}
	Let $G$ be a very well–covered graph with $2n$ vertices. Then we have $\dim(S/NI(G))\geq n$.  If in addition, $G$ is
	well-dominated, then $\dim(S/NI(G))=n$. 
\end{Lemma}

\begin{proof}
	Since $G$ is very well–covered, any maximal independent set of $G$ has cardinality $n$. So it follows that $\height(NI(G))=\gamma(G)\leq \alpha(G)=n$ and hence the first inequality holds. If further $G$ is well-dominated , then  $\gamma(G)$ is equal to the cardinality of any maximal independent set of $G$. So we have $\dim(S/NI(G))=n$. 	
\end{proof}

In order to study the Cohen-Macaulayness of $S/NI(G)$, one may always reduce to the case that $G$ is a connected graph. In fact, let $G$ be a graph with the connected components $G_1,\ldots,G_r$. Then $NI(G)=\sum_{i=1}^r NI(G_i)$, where $NI(G_i)\subset S_i=K[V(G_i)]$ are ideals in disjoint polynomial rings. Then  $S/NI(G)=\bigotimes_{i=1}^rS_i/NI(G_i)$. Hence, $S/NI(G)$ is Cohen-Macaulay if and only if $S_i/NI(G_i)$ is Cohen-Macaulay for any $1\leq i\leq r$. Thus in Theorem \ref{VW-CM} we may assume that $G$ is connected.

Now, we are in a position to prove the main result of this section.

\begin{Theorem}\label{VW-CM}
	Let $G$ be a connected very well-covered graph with $2n\geq 4$ vertices. The following are equivalent.
	\begin{enumerate}
		\item [(i)] $S/NI(G)$ is Cohen-Macaulay. 
		
		\item [(ii)] $G$ is well-dominated.
		
		\item [(iii)] $G=C_4$, or $G$ is a whisker graph.  
		
		\item [(iv)] If $G\neq C_4$, then $NI(G)$ is a complete intersection.
		
		\item [(v)] If $G\neq C_4$, then $S/NI(G)$ is Gorenstein.
		
		\item [(vi)] $\gamma(G)=\alpha(G)=n$.
	\end{enumerate}	  
\end{Theorem}

\begin{proof} 
	(i) $\Rightarrow$ (ii): Since $S/NI(G)$ is Cohen-Macaulay, $NI(G)$ is an unmixed ideal, that is all of its minimal prime ideals have the same height. Then (ii) follows from \cite[Lemma 2.2.]{SM}.  
	\medskip
	
	(ii) $\Rightarrow$ (iii):  By Lemma \ref{part} there exists a partition of the vertex set as $V(G)=X\sqcup Y$, such that $X =\{x_1,\ldots,x_n\}$ is minimal vertex cover of $G$, $Y =\{y_1,\ldots,y_n\}$ is a maximal independent set of $G$ and $\{\{x_1,y_1\},\ldots,\{x_n,y_n\}\}\subseteq E(G)$.
	Since $G$ is well-dominated, by Lemma \ref{dimVW} we have $\dim(S/NI(G))=n$. So applying Proposition \ref{VWWD} we conclude that every spanning tree of $G$ is a whisker graph. We need to show that $G$ is either a whisker graph, or $G=C_4$.
	If for any $1\leq i\leq n$ we have either $\deg_G(x_i)=1$ or
	$\deg_G(y_i)=1$, then $G$ has precisely $n$ vertices of degree one, which are adjacent to  distinct $n$ vertices in $G$. Hence, $G$ is a whisker graph, and in this case we are done. 
	
	So we may assume that there exists $1\leq j\leq n$ such that  $\deg_G(x_j)>1$ and  $\deg_G(y_j)>1$. We show that $G=C_4$. 
	Without loss of generality, we may assume that 
	$\deg_G(x_1)>1$ and  $\deg_G(y_1)>1$.
	Since $\deg_G(y_1)>1$ and $Y$ is an independent set of $G$, there exists $t\neq 1$ such that $x_t\in N_G(y_1)$.
	We put weights on the edges of $G$
	by assigning weight $1$
	to the edges $e_i=\{x_i,y_i\}$, for $1\leq i\leq n$, and $\{y_1,x_t\}$ and assigning
	weight $2$
	to all the other edges of $G$. Now applying Kruskal’s algorithm we obtain a spanning tree $T$ of $G$ such that $$\{\{y_1,x_t\}\}\cup\{\{x_1,y_1\},\ldots,\{x_n,y_n\}\}\subseteq E(T).$$
	
	By Proposition \ref{VWWD}, $T$ is a whisker graph. So any vertex in $T$ which is not of degree one, is adjacent to a vertex of degree one (leaf) of $T$.
	Since $x_1,x_t\in N_T(y_1)\subseteq X$, we have $\deg_T(y_1)\geq 2$. Hence, $y_1$ is adjacent to a vertex $x_\ell\in X$ in the graph $T$ which is a leaf of $T$. Then have $y_1,y_\ell\in N_T(x_\ell)$. This together with $\deg_T(x_\ell)=1$ implies that $\ell=1$. 
	Hence, $x_1$ is a leaf of $T$. Since by our assumption $\deg_G(x_1)>1$, there exists a vertex $z\in V(G)$ with $z\neq y_1$ such that $\{x_1,z\}\in E(G)$ and $\{x_1,z\}\notin E(T)$.
	
	We claim that $\{y_1,z\}\notin E(T)$. On the contrary, suppose that $\{y_1,z\}\in E(T)$. Then $z=x_p$ for some $p\neq 1$ and  $x_1,y_1,x_p$ form a triangle in $G$. Moreover, $T'=(T\setminus \{x_1,y_1\})\cup \{{\{x_1,x_p\}}\}$ is 
	a connected graph with $|E(T')|=|E(T)|=|V(G)|-1$. Hence, $T'$ is also a spanning tree of $G$ with $N_{T'}(x_1)=\{x_p\}$. Thus  $\{x_1\}\cup Y$ forms an independent set of $T'$. This contradicts to the fact that $T'$ is a whisker graph. Thus $\{y_1,z\}\notin E(T)$, as was claimed.    
	Let $$y_1=w_0,w_1,\ldots,w_k,w_{k+1}=z$$ be the unique path in $T$  connecting the vertices $y_1$ and $z$, where $k\geq 1$. Then $x_1,y_1,w_1,\ldots,w_k,w_{k+1}=z$ is a path in $T$ and adding the edge $\{x_1,z\}$ to this path gives a cycle in $T\cup \{{\{x_1,z\}}\}$. Hence, $T_1=(T\setminus \{w_k,z\})\cup \{{\{x_1,z\}}\}$ is again a spanning tree of $G$. Since $T_1$ is a whisker tree and $N_{T_1}(x_1)=\{y_1,z\}$, it follows that $x_1$ is adjacent to a leaf in $T_1$ which is either $y_1$ or $z$.   Notice that $y_1$
	is not a leaf  in $T_1$, as $x_1,w_1\in N_{T_1}(y_1)$. Hence, we conclude that $\deg_{T_1}(z)=1$. This implies that $\deg_{T}(z)=1$ and $N_T(z)=\{w_k\}$. Now, consider $T_2=(T\setminus \{w_{k-1},w_k\})\cup \{{\{x_1,z\}}\}$ which is again a spanning tree of $G$. We have $N_{T_2}(x_1)=\{y_1,z\}$. Since $T_2$ is a whisker tree,   $x_1$ is adjacent in $T_2$ to a leaf of $T_2$. Since $N_{T_2}(z)=\{x_1,w_k\}$, we conclude that $y_1$ is a leaf of $T_2$. Using this and the fact that $\{x_1,w_1\}\subseteq N_T(y_1)$, we deduce that $\{w_{k-1},w_k\}=\{y,w_1\}$ and  $N_T(y_1)=\{x_1,w_1\}$. Therefore, $k=1$ and $L: x_1,y_1,w_1,z$ is a path in $T$. As was shown above,   $\deg_T(x_1)=\deg_T(z)=1$ and $\deg_T(y_1)=2$. Since $T_2$ is a whisker graph and $z$ is not a leaf in $T_2$, its neighbor $w_1$ is a leaf in $T_2$ (note that $z$ is adjacent to only $w_1$ and $x_1$ in $T_2$ and $x_1$ is not a leaf of $T_2$). Thus  
	$\deg_{T_2}(w_1)=1$, which implies that 	$\deg_{T}(w_1)=2$.      
	Now,  $L: x_1,y_1,w_1,z$ is a path in $T$ with  $\deg_T(x_1)=\deg_T(z)=1$ and $\deg_T(y_1)=\deg_{T}(w_1)=2$. Hence $L$ is a connected component of $T$. Since $T$ is connected we conclude that $T=L$. So $G$ is a graph on $V(G)=\{x_1,y_1,w_1,z\}$ with the path $T=L$ as its spanning tree. Note that $G\neq T$, since $\deg_G(x_1)>1=\deg_T(x_1)$. So there exists an edge $e\in E(G)\setminus E(T)$. Notice that if $e=\{x_1,w_1\}$ or $e=\{y_1,z\}$, then $G$ has a spanning tree which is not a whisker graph, and this is a contradiction. So $e=\{x_1,z\}$ is the only edge in $E(G)\setminus E(T)$, which means that $G=C_4$.

	\medskip
	
	(iii) $\Rightarrow$ (iv): Let $V(G)=\{x_1,\ldots,x_n\}\cup\{y_1,\ldots,y_n\}$ with $x_1,\ldots,x_n$ of degree one and $\{x_i,y_i\}\in E(G)$ for all $i$. Then clearly, $NI(G)=(x_1y_1,\ldots,x_ny_n)$ and $x_1y_1,\ldots,x_ny_n$ is a regular $S$-sequence. 
	
	\medskip
	
	(iv) $\Rightarrow$ (v) and (v) $\Rightarrow$ (i): 
The implications are well-known facts for any ideal in $S$. Notice that if $G=C_4$, then $\dim(S/NI(G))=4-\gamma(G)=2$ and by Proposition \ref{lPd}(a), we have $\depth(S/NI(G))=\tau(G)=2$. So 
$S/NI(G)$ is Cohen-Macaulay.

(ii) $\Rightarrow$ (vi): Since $G$ is well-dominated, any maximal independent set of $G$ has cardinality $\gamma(G)$. Since $\alpha(G)=n$, the desired equality holds.

\medskip

(vi) $\Rightarrow$ (ii): We have $\alpha(G)=n$. Thus by assumption,
$\gamma(G)=n$. Let $D$ be an arbitrary minimal dominating set of $G$. We show that $|D|=n$. To this aim, first we show that $D^c=V(G)\setminus D$ is a dominating set of $G$. By contradiction, assume that there exists a vertex $z\in V(G)$ such that $N_G[z]\cap D^c=\emptyset$. Then $N_G[z]\subseteq D$. This implies that $D\setminus\{z\}$ is a  dominating set of $G$, which contradicts to the minimality of $D$. So $D^c$ is a dominating set, as claimed. Hence, $n=\gamma(G)\leq |D^c|=2n-|D|$. Thus $|D|\leq n=\gamma(G)\leq |D|$. So $|D|=n$. Since $D$ was chosen to be an arbitrary minimal dominating set of $G$, we conclude that $G$ is well-dominated.   
\end{proof}

Note that any well-covered bipartite graph without isolated vertices is a very well-covered graph.  So as a corollary of Theorem \ref{VW-CM}, we recover a characterization of bipartite graphs $G$ for which $S/NI(G)$ is Cohen-Macaulay, proved in \cite[Theorem 5.3.6]{L}.  We also add the equivalent condition $\gamma(G)=\alpha(G)$ to the statement of \cite[Theorem 5.3.6]{L}, using a result proved by Topp and Volkmann \cite{TV}. As Remark \ref{CM} shows, this equality is in general only a necessary condition for the Cohen-Macaulayness of $S/NI(G)$.  

\begin{Remark}\label{CM}
If $S/NI(G)$ is Cohen-Macaulay, then $NI(G)$ is an unmixed ideal, and hence $G$ is well-dominated. Thus any maximal independent set of $G$ has cardinality equal to $\gamma(G)$. Therefore,  $\gamma(G)=\alpha(G)$. On the other hand, for the graph $G=C_5$ we have $\gamma(G)=\alpha(G)=2$, while
 $G$ is not Cohen-Macaulay. Indeed, $\depth(S/NI(G))=2<\dim(S/NI(G))=3$, see Remark  \ref{C5}.	  
\end{Remark}

\begin{Corollary}\label{bipartiteCM}
Let $G$ be a connected bipartite graph with $n\geq 4$ vertices. Then the following are equivalent.
	\begin{enumerate}
		\item [(i)] $S/NI(G)$ is Cohen-Macaulay. 
		
		\item [(ii)] $G$ is well-dominated.  
		
		 \item [(iii)] $G=C_4$, or $G$ is a whisker graph.     

		\item [(iv)]  If $G\neq C_4$, then $NI(G)$ is a complete intersection.  
		
		\item [(v)]  If $G\neq C_4$, then $S/NI(G)$ is Gorenstein.
			
		\item [(vi)] $\gamma(G)=\alpha(G)$.  
	\end{enumerate}	  
\end{Corollary}

\begin{proof}
(i) $\Rightarrow$ (ii): The ideal $NI(G)$ is unmixed. So $G$ is well-dominated.  
 
(ii) $\Rightarrow$ (iii): Since $G$ is well-dominated, it is a well-covered bipartite graph, see  Remark \ref{covered}. So it is a very well-covered graph. Thus the result holds by Theorem \ref{VW-CM}.

(iii) $\Rightarrow$ (iv): is proved in Theorem \ref{VW-CM}.

(iv) $\Rightarrow$ (v) and (v) $\Rightarrow$ (i): 
hold for any ideal in $S$. As was shown in the proof of Theorem \ref{VW-CM}, $S/NI(C_4)$ is Cohen-Macaulay.

(i) $\Rightarrow$ (vi): holds by Remark \ref{CM}.

(vi) $\Rightarrow$ (iii):  holds by \cite[Theorem 2]{TV}. 
\end{proof}

In \cite[Theorem 3]{TV}, Topp and Volkmann provided a characterization of connected block graphs $G$ satisfying the equality $\gamma(G)=\alpha(G)$. We notice that for chordal graphs this problem is closely related to the Cohen-Macaulayness of the ring $S/NI(G)$. A characterization of chordal graphs $G$ for which $S/NI(G)$  is Cohen-Macaulay was given in \cite[Theorem 5.4.9]{L}. Using this characterization, in Corollary \ref{alphaGamma} we show that this in fact gives a characterization of chordal graphs $G$ which satisfy $\gamma(G)=\alpha(G)$, generalizing the result of Topp and Volkmann to chordal graphs. 

\begin{Theorem}\cite[Theorem 5.4.9]{L}\label{chordalCM}
	Let $G$ be a chordal graph. Then the following conditions are equivalent.
\begin{enumerate}
	\item [(i)]  $S/NI(G)$ is Cohen-Macaulay.   
	
	\item [(ii)] $G$ is well-dominated. 
	
	\item [(iii)] $V(G)$ is the disjoint union of maximal cliques of $G$  which admit a free vertex.   
	
\item [(iv)]  $NI(G)$ is a complete intersection. 
\end{enumerate}	 	
\end{Theorem}

\begin{Corollary}\label{alphaGamma}
	Let $G$ be a chordal graph. Then the following conditions are equivalent.
	\begin{enumerate}
		\item [(i)] 	$\gamma(G)=\alpha(G)$.
		
		\item [(ii)] $G$ is well-covered.  
		
		\item [(iii)] $G$ is well-dominated.     
		
		\item [(iv)] $V(G)$ is the disjoint union of maximal cliques of $G$  which admit a free vertex.  
	\end{enumerate}	 
\end{Corollary}

\begin{proof}
(iv) $\Rightarrow$ (iii): hold by Theorem \ref{chordalCM}.

\medskip

(iii) $\Rightarrow$ (ii): holds by Remark \ref{covered}.   

\medskip

(ii) $\Rightarrow$ (iv): follows from \cite[Theorem 2.1]{HHZ1}. 

\medskip
	
(iii) $\Rightarrow$ (i): By the assumption, all minimal dominating sets of $G$ have the same cardinality, say $k$. So any maximal independet set of $G$ has cardinality $k$. Thus 	$\gamma(G)=\alpha(G)=k$. 

\medskip
	
	(i) $\Rightarrow$ (ii):  Let $B$ be a maximal independent set of $G$ with $\alpha(G)=|B|$, and let $A$ be an arbitrary maximal independent set of $G$. Since $A$ is a minimal dominating set of $G$, we have $$|B|=\alpha(G)=\gamma(G)\leq |A|\leq \alpha(G).$$ So $|A|=|B|=\alpha(G)$, for any maximal independent set $A$ of $G$.  
\end{proof}


\end{document}